\documentclass[12pt]{amsart}
\usepackage[margin=1in]{geometry}
\usepackage[english]{babel}
\usepackage[utf8]{inputenc}
\usepackage{subcaption}
\usepackage{amsmath}
\usepackage{amssymb}
\usepackage{amsfonts}
\usepackage{amsthm}
\usepackage{mathdots}
\usepackage{mathrsfs}
\usepackage[all]{xy}
\usepackage[pdftex]{graphicx}
\usepackage{color}
\usepackage{cite}
\usepackage{url}
\usepackage{indent first}
\usepackage[labelfont=bf,labelsep=period,justification=raggedright]{caption}
\usepackage[english]{babel}
\usepackage[utf8]{inputenc}
\usepackage{hyperref}
\usepackage[colorinlistoftodos]{todonotes}
\usepackage{tkz-fct}
\usepackage{tikz}
\usetikzlibrary{calc}
\usepackage{pgfplots}
\usepackage{multicol}
\PassOptionsToPackage{dvipsnames,svgnames}{xcolor}
\usepackage{textcomp}
\usepackage{bbm}
\usepackage[ruled,vlined]{algorithm2e}
\usepackage[shortlabels]{enumitem}
\usepackage{stackrel}

\newcommand{\A}{\mathbf{A}}
\newcommand{\M}{\mathbf{M}}

\theoremstyle{plain}
\newtheorem{thm}{Theorem}
\newtheorem{lemma}[thm]{Lemma}

\theoremstyle{definition}

\theoremstyle{remark}

\numberwithin{equation}{section}
\numberwithin{thm}{section}

\begin{document}
\title{Infinite matrix products and hypergeometric zeta series}
\author{T. Wakhare$^{1}$ and C. Vignat$^{2}$}
	\address{$^{1}$ Department of Electrical Engineering and Computer Science, Massachusetts Institute of Technology, Cambridge, Massachusetts, USA}
	\email{twakhare@mit.edu}
\address{$^{2}$ Department of Mathematics, Tulane University, New Orleans, Louisiana, USA}
	\email{cvignat@tulane.edu}

\maketitle

\begin{abstract}
An unpublished identity of Gosper restates a hypergeometric identity for odd zeta values in terms of an infinite product of matrices. We show that this correspondence runs much deeper, and show that many examples of WZ-accelerated series for zeta values lift to infinite matrix products. We also introduce a new matrix subgroup, the Gosper group, which all of our matrix products fall into.
\end{abstract}

\section{Introduction}

%
%
%
%

In his famous book ``Mathematical Constants'' \cite{Finch}, Finch cites
an unpublished result by Gosper \cite{Gosper}:
\begin{equation}
\prod_{k=1}^{\infty}\left[\begin{array}{cc}
-\frac{k}{2\left(2k+1\right)} & \frac{5}{4k^{2}}\\
0 & 1
\end{array}\right]=\left[\begin{array}{cc}
0 & \zeta\left(3\right)\\
0 & 1
\end{array}\right],\label{eq:product 2}
\end{equation}
and its $\left(N+1\right)\times\left(N+1\right)$ extension, for $N\geqslant2,$
\begin{equation}
\prod_{k=1}^{\infty}\left[\begin{array}{cccccc}
-\frac{k}{2\left(2k+1\right)} & \frac{1}{2k\left(2k+1\right)} & 0 & \dots & 0 & \frac{1}{k^{2N}}\\
0 & -\frac{k}{2\left(2k+1\right)} & \frac{1}{2k\left(2k+1\right)} & \dots &  & \frac{1}{k^{2N-2}}\\
\vdots & \vdots & \vdots &   & \vdots & \vdots\\
0 & 0 & 0 & \dots & \frac{1}{2k\left(2k+1\right)} & \frac{1}{k^{4}}\\
0 & 0 & 0 & \dots & -\frac{k}{2\left(2k+1\right)} & \frac{5}{4k^{2}}\\
0 & 0 & 0 & \dots & 0 & 1
\end{array}\right]=\left[\begin{array}{cccc}
0 & \dots & 0 & \zeta\left(2N+1\right)\\
0 & \dots & 0 & \zeta\left(2N-1\right)\\
\vdots &  & \vdots & \vdots\\
0 & \dots & 0 & \zeta\left(5\right)\\
0 & \dots & 0 & \zeta\left(3\right)\\
0 & \dots & 0 & 1
\end{array}\right].\label{eq:product n}
\end{equation}
We will show that this formula is in fact equivalent to Koecher's identity \cite[Eq. (3)]{Koecher}
\begin{equation}\label{koecher}
    \sum_{n=0}^\infty \frac{1}{n(n^2-x^2)} = \frac12 \sum_{k=1}^\infty \frac{(-1)^{k-1}}{\binom{2k}{k} k^3} \frac{5k^2-x^2}{k^2-x^2} \prod_{m=1}^{k-1} \left(.1-\frac{x^2}{m^2} \right).
\end{equation}
By extracting coefficients of $1$ and $x^2$ in Koecher's identity, we recover Markov's series acceleration identity \cite{Markov}
\[
\zeta\left(3\right)=\frac{5}{2}\sum_{n\geqslant1}\frac{\left(-1\right)^{n-1}}{n^{3}\binom{2n}{n}}
\]
and its higher order counterpart
\[
\zeta\left(5\right)=2\sum_{n=1}^{\infty}\frac{\left(-1\right)^{n-1}}{n^{5}\binom{2n}{n}}-\frac{5}{2}\sum_{n=1}^{\infty}\frac{\left(-1\right)^{n-1}H_{n-1}^{\left(2\right)}}{n^{3}\binom{2n}{n}}.
\]

These are efficiently encoded by the matrix product. By extracting other coefficients of $x^n$ in Koecher's identity, we recover counterparts for $\zeta(2n+1)$ which are again encoded by the matrix product.

This correspondence runs much deeper, and we will show that several hypergeometric-type series for the zeta function at small integers are equivalent to infinite products for $N \times N$ matrices. The fact that these identities support an expression in terms of matrix products is already interesting. The pattern of entries of some small matrices suggest the general form of the relevant $n \times n$ generalizations, which would then be equivalent to new accelerated series for zeta values.


\section{Background}
\subsection{Special Functions}
The \textit{Riemann zeta function}, absolutely convergent for $s\in\mathbb{C},\Re s >1$ is given by
\begin{equation}
    \zeta(s):= \sum_{n=1}^\infty \frac{1}{n^s}.
\end{equation}
This straightforwardly extends to the \textit{Hurwitz zeta function} with the addition of a parameter $z\in\mathbb{C},z \neq 0,-1,-2,\ldots$:
\begin{equation}
    \zeta(s|z) := \sum_{n=1}^\infty \frac{1}{(n+z)^s},
\end{equation}
so that $\zeta(s)=\zeta(s|1)$.

The \textit{harmonic numbers} are given by $H_0:=0$ 
and
\begin{equation}
    H_n:=\sum_{k=1}^n \frac{1}{k},\quad  n\geqslant 1.
\end{equation}
The \textit{hyper-harmonic numbers} are defined similar. 
We will also consider the \textit{elementary symmetric functions}
\begin{equation}\label{esym}
    e_\ell^{(s)}(k):= [t^\ell] \prod_{j=1}^{k-1} \left(1+\frac{t}{j^s}\right) = \sum_{1 \leqslant j_1 < j_2 < \cdots < j_\ell \leqslant k-1} \frac{1}{(j_1\cdots j_\ell)^s},
\end{equation}
which reduce to the harmonic numbers at $e_1^{1}(n) = H_{n-1}$.


\section{The Gosper Group}

Each Gosper matrix in the product (\ref{eq:product n}) has the form
\[
\M_{k}=\left[\begin{array}{cc}
\A_{k} & \mathbf{u}_{k}\\
\mathbf{0} & 1
\end{array}\right]
\]
where $\A_{k}$ is square $\left(N\times N\right)$, $\mathbf{u}_{k}$ is a
$\left(N\times1\right)$ vector and $\mathbf{0}$ is the $\left(1\times N\right)$ vector of zeros. Matrices of this kind form a group, which we shall name the \textbf{Gosper group}. With $\mathbf{I}_N$ the $(N\times N)$ identity matrix, the unit element of the group is $\left[\begin{array}{cc}
\mathbf{I}_{N} & \mathbf{0}\\
\mathbf{0} & 1
\end{array}\right]$, and the inverse of an element $\M=\left[\begin{array}{cc}
\A & \mathbf{u}\\
\mathbf{0} & 1
\end{array}\right]
$ is $\M^{-1}=\left[\begin{array}{cc}
\A^{-1} & -\A^{-1}\mathbf{u}\\
\mathbf{0} & 1
\end{array}\right]
$. Closure follows from $\M_{1}\M_{2}= \left[\begin{array}{cc}
\A_{1}\A_{2} & \A_{1}\mathbf{u}_{2}+\mathbf{u}_{1}\\
\mathbf{0} & 1
\end{array}\right].
$
We can inductively verify that
\[
\M_{1}\M_{2}\dots \M_{n}=\left[\begin{array}{cc}
\A_{1}\A_{2}\dots \A_n & \sum_{k=1}^{n}\A_{1}\dots \A_{k-1}\mathbf{u}_{k}\\
\mathbf{0} & 1
\end{array}\right].
\]

\subsection{Toeplitz Matrices}

Moreover, each $A_{k}$ in Gosper's identity has the simple form 
\[
\A_{k}=\alpha_{k}\mathbf{I}+\beta_{k}\mathbf{J}
\]
where $\mathbf{J}$ is the $\left(N\times N\right)$ matrix with a first superdiagonal
of ones. 

Hence $\mathbf{J}^{N}=0$ and, for $p\geqslant  N$, we have
\begin{align*}
\A_{1}\A_{2}\ldots \A_{p}  &=\left(\alpha_{1}\mathbf{I}+\beta_{1}\mathbf{J}\right)\left(\alpha_{2}\mathbf{I}+\beta_{2}\mathbf{J}\right)\ldots\left(\alpha_{p}\mathbf{I}+\beta_{p}\mathbf{J}\right)\\&=\left(\prod_{i=1}^{p}\alpha_{i}\right)\left(\mathbf{I}+\sum_{j=1}^{p}\frac{\beta_{j}}{\alpha_{j}}\mathbf{J}+\dots+\sum_{1\leqslant j_{1}<\dots<j_{N-1}\leqslant p}\frac{\beta_{j_{1}}\ldots\beta_{j_{N-1}}}{\alpha_{j_{1}}\ldots\alpha_{j_{N-1}}}\mathbf{J}^{N-1}\right).
\end{align*}
For $p<N$ the summation is instead truncated at $\mathbf{J}^p$.

The general form of the components of the limiting infinite product case can be deduced by induction.
\begin{lemma}\label{mainlem}
The components of 
\begin{equation}
\prod_{k=1}^{\infty}\left[\begin{array}{cc}
\A_k & \mathbf{u}_k\\
\mathbf{0} & 1
\end{array}\right]=\left[\begin{array}{cc}
\prod_{k=1}^\infty \A_k & \mathbf{v}_\infty\\
\mathbf{0} & 1
\end{array}\right],\label{eq:product 3}
\end{equation}
with
\[
\left[v_{\infty}^{\left(N\right)},\dots,v_{\infty}^{\left(1\right)}\right]^{T}:= \mathbf{v}_{\infty}=\sum_{p=1}^\infty\A_{1}\dots \A_{p-1}\mathbf{u}_{p},
\]
are
\begin{align*}
v_{\infty}^{\left(1\right)}&=\sum_{p=1}^\infty\left(\alpha_{1}\cdots\alpha_{p-1}\right)u_{p}^{\left(1\right)}, \\
v_{\infty}^{\left(2\right)}&=\sum_{p=1}^\infty\left(\alpha_{1}\cdots\alpha_{p-1}\right)\left(u_{p}^{\left(2\right)}+\left(\sum_{j=1}^{p-1}\frac{\beta_{j}}{\alpha_{j}}\right)u_{p}^{\left(1\right)}\right), \\
&\vdots \nonumber\\
v_{\infty}^{\left(\ell\right)}&=\sum_{p=1}^\infty\left(\alpha_{1}\cdots\alpha_{p-1}\right)\left(u_{p}^{\left(\ell\right)}+\left(\sum_{j=1}^{p-1}\frac{\beta_{j}}{\alpha_{j}}\right)u_{p}^{\left(\ell-1\right)}+\dots+\left(\sum_{1\leqslant j_{1}<\dots<j_{\ell-1}\leqslant p-1}\frac{\beta_{j_{1}}\dots\beta_{j_{\ell-1}}}{\alpha_{j_{1}}\dots\alpha_{j_{\ell-1}}}\right)u_{p}^{\left(1\right)}\right),
\end{align*}
with $1 \leqslant \ell \leqslant N$.
\end{lemma}

Already the connection to zeta series and hyperharmonic numbers is clear: with the correct choice of $\alpha$ and $\beta$, the multiple sums will reduce to multiple zeta type functions.

These matrix products also exhibit a stability phenomenon, where increasing the dimension of the matrix does not impact any entries in $\mathbf{v}_\infty$ except the top right one, since mapping $N\to N+1$ only changes the formula for $v_\infty^{(N+1)}$.

We will consistently refer to the $N=1$ and $N=2$ cases. Explicitly, when $N=1$ so that both $A_k$ (denoted $\alpha_k$  to avoid confusion) and $u_k$ are scalars, we have 

\begin{lemma}
For $N=1,$
\begin{equation}
\prod_{k=1}^{n}\left[\begin{array}{cc}
\alpha_k & \beta_k\\
0 & 1
\end{array}\right]=\left[\begin{array}{cc}
\prod_{k=1}^n \alpha_k & 
\sum_{k=1}^n \alpha_1\dots \alpha_{k-1}\beta_k\\
0 & 1
\end{array}\right].
\end{equation}
\end{lemma}
Although we will only need the $n\to \infty$ limit, let us note that this identity holds for finite $n$.

\section{Koecher' Identity}

\begin{thm}
Identity \eqref{eq:product 2} and Koecher's identity are equivalent.
\end{thm}
\begin{proof}
Begin with Koecher's identity \eqref{koecher}. By extracting coefficients of $x^{2n}$, in general we obtain 
\begin{equation}\label{koechercoef}
\zeta(2n+3) = \frac{5}{2}\sum_{k=1}^\infty\frac{\left(-1\right)^{k-1}}{k^{3}\binom{2k}{k}}(-1)^n e^{(2)}_n(k) + 2 \sum_{j=1}^n \sum_{k=1}^\infty \frac{(-1)^{k-1}}{k^{2j+3}\binom{2k}{k}} (-1)^{n-j} e_{n-j}^{(2)}(k).
\end{equation}

Take $\alpha_k = -\frac{k}{2(2k+1)}$, $\beta_k = \frac{1}{2k(2k+1)}$, $u_k^{(1)} = \frac{5}{4k^2}$, and $u_k^{(\ell)} = \frac{1}{k^{2\ell+2}}$ for $2\leqslant \ell \leqslant N$. This corresponds to the Gosper matrix
$$\left[\begin{array}{cc}
\A_k & \mathbf{u}_k\\
\mathbf{0} & 1
\end{array}\right]=\left[\begin{array}{cccccc}
-\frac{k}{2\left(2k+1\right)} & \frac{1}{2k\left(2k+1\right)} & 0 & \dots & 0 & \frac{1}{k^{2N}}\\
0 & -\frac{k}{2\left(2k+1\right)} & \frac{1}{2k\left(2k+1\right)} & \dots &  & \frac{1}{k^{2N-2}}\\
\vdots & \vdots & \vdots &   & \vdots & \vdots\\
0 & 0 & 0 & \dots & \frac{1}{2k\left(2k+1\right)} & \frac{1}{k^{4}}\\
0 & 0 & 0 & \dots & -\frac{k}{2\left(2k+1\right)} & \frac{5}{4k^{2}}\\
0 & 0 & 0 & \dots & 0 & 1
\end{array}\right].
$$

Then
$$ \prod_{i=1}^p \alpha_i = {(-1)^{p}}\prod_{i=1}^p\frac{i^2}{(2i)(2i+1)} = (-1)^{p} \frac{p!^2}{(2p+1)!},$$
and (for $2\leqslant \ell \leqslant N$)
$$\sum_{j_{1}<\dots<j_{\ell-1}\leqslant p-1}\frac{\beta_{j_{1}}\dots\beta_{j_{\ell-1}}}{\alpha_{j_{1}}\dots\alpha_{j_{\ell-1}}} = (-1)^{\ell}\sum_{j_{1}<\dots<j_{\ell-1}\leqslant p-1} \frac{1}{(j_1 \cdots j_{\ell-1})^2} = (-1)^{\ell} e_{\ell-1}^{(2)}(p). $$
We deduce
$$\lim_{p\to \infty} \alpha_1\cdots \alpha_p = 0,$$ while 
$$\lim_{p\to \infty} \sum_{j_{1}<\dots<j_{k}\leqslant p-1} \frac{1}{(j_1 \cdots j_{k})^2} \leqslant \lim_{p\to \infty} \sum_{j_1=1}^p \frac{1}{j_1^2} = \zeta(2).$$
Hence, applying Lemma \ref{mainlem}, we deduce 
$$\prod_{i=1}^\infty \A_i = \bf{0}.$$
The components in the right column are then explicitly given as 
\begin{align*}
v_{\infty}^{\left(\ell\right)}&=\sum_{p=1}^\infty\left(\alpha_{1}\cdots\alpha_{p-1}\right)\left(u_{p}^{\left(\ell\right)}+\left(\sum_{j=1}^{p-1}\frac{\beta_{j}}{\alpha_{j}}\right)u_{p}^{\left(\ell-1\right)}+\dots+\left(\sum_{1\leqslant j_{1}<\dots<j_{\ell-1}\leqslant p-1}\frac{\beta_{j_{1}}\dots\beta_{j_{\ell-1}}}{\alpha_{j_{1}}\dots\alpha_{j_{\ell-1}}}\right)u_{p}^{\left(1\right)}\right) \\
&= \sum_{p=1}^\infty \frac{(-1)^{p-1}(p-1)!^2}{(2p-1)!} \left( \frac{1}{p^{2\ell+2}} - \frac{e_1^{(2)}(p)}{p^{2\ell}} + \cdots + (-1)^{\ell-1}\frac{5}{4} \frac{e^{(2)}_{\ell-1}(p)}{p^2} \right) \\
&= \frac{5}{2}\sum_{p=1}^\infty \frac{(-1)^{p-1}}{p^3 \binom{2p}{p}}e_{\ell-1}^{(2)}(p) + 2\sum_{j=1}^{\ell-1}  \sum_{p=1}^\infty \frac{(-1)^{p-1}}{p^{3+2j} \binom{2p}{p}}e_{\ell-1-j}^{(2)}(p) (-1)^{\ell-1-j}.
\end{align*}
We see that this is exactly the formula from Koecher's identity, hence equals $\zeta(2\ell+1)$ for $1\leqslant \ell \leqslant N$. 
\end{proof}

\section{Leschiner's identity}
Begin with the Leschiner identity
\[
\sum_{n\geqslant1}\frac{\left(-1\right)^{n-1}}{n^{2}-z^2}=\frac{1}{2}\sum_{k\geqslant1}\frac{1}{\binom{2k}{k}k^{2}}\frac{3k^{2}+z^2}{k^{2}-z^2}\prod_{j=1}^{k-1}\left(1-\frac{z^2}{j^{2}}\right),
\]
so that
\[
\tilde{\zeta}\left(2\right)=\frac{3}{2}\sum_{k\geqslant1}\frac{1}{\binom{2k}{k}k^{2}},
\]
and
\[
\bar{\zeta}\left(4\right)=\frac{3}{2}\sum_{k\geqslant1}\frac{1}{\binom{2k}{k}k^{2}}\left[\frac{4}{k^{2}}-H_{k-1}^{\left(2\right)}\right],
\]
and in general (I think I made a mistake here)
$$
\tilde{\zeta}(2n+2) = \frac{3}{2}\sum_{k=1}^\infty\frac{1}{k^{2}\binom{2k}{k}}(-1)^n e^{(2)}_n(k) +6 \sum_{j=1}^n \sum_{k=1}^\infty \frac{1}{k^{2j+2}\binom{2k}{k}} (-1)^{n-j} e_{n-j}^{(2)}(k).
$$
A Gosper representation for $\bar{\zeta}\left(2\right)$ and $\bar{\zeta}\left(4\right)$
is
\[
\prod_{n\geqslant1}\left(\begin{array}{ccc}
\frac{n}{2\left(2n+1\right)} & \frac{-1}{2n\left(2n+1\right)} & \frac{1}{n^{3}}\\
0 & \frac{n}{2\left(2n+1\right)} & \frac{3}{4n}\\
0 & 0 & 1
\end{array}\right)=\left(\begin{array}{ccc}
0 & 0 & \bar{\zeta}\left(4\right)\\
0 & 0 & \bar{\zeta}\left(2\right)\\
0 & 0 & 1
\end{array}\right).
\]
This will generalize using the same method as Koecher.
\section{Borwein's Identity}
\subsection{the infinite product case}
Extracting coefficient of $z^{2n}$ from Borwein's identity \cite{Borwein}
\begin{equation}
\sum_{n\geqslant1}\frac{1}{n^{2}-z^2}=3\sum_{k\geqslant1}\frac{1}{\binom{2k}{k}}\frac{1}{k^{2}-z^2}\prod_{j=1}^{k-1}\frac{j^{2}-4z^2}{j^{2}-z^2}.\label{eq:Borwein}
\end{equation}
gives 
\begin{align*}
    \sum_{k\geqslant1}\frac{1}{\binom{2k}{k}}\frac{1}{k^{2}-z^2}\prod_{j=1}^{k-1}\frac{j^{2}-4z^2}{j^{2}-z^2} &= \sum_{k\geqslant 1} \frac{1}{k^2\binom{2k}{k}} \prod_{j=1}^{k-1} \left(1-\frac{4z^2}{j^2} \right) \prod_{j=1}^k \frac{1}{1-\frac{z^2}{j^2}} \\
    &= \sum_{k\geqslant 1} \frac{1}{k^2\binom{2k}{k}} \sum_{\ell \geqslant 0} z^{2\ell} 4^\ell e_\ell^{(2)}(k) \sum_{m\geqslant 0} z^{2m} h_{m}^{(2)}(k+1),
\end{align*}
where $h_m$ is the complete symmetric function. This gives us a formula for the coefficient of $z^{2n}$ as a convolution over $h_m$ and $e_m$. How do we encode this in the matrix, in terms of $\alpha_k,\beta_k, u_k$?

\begin{thm}
A Gosper representation for $\zeta\left(2\right)$ is obtained as
\[
\prod_{n\geqslant1}\left(\begin{array}{cc}
\frac{n}{2\left(2n+1\right)} & \frac{3}{2n}\\
0 & 1
\end{array}\right)=\left(\begin{array}{cc}
0 & \zeta\left(2\right)\\
0 & 1
\end{array}\right).
\]
\end{thm}

\begin{proof}
Identifying the constant term produces
\[
\zeta\left(2\right)=3\sum_{k\geqslant1}\frac{1}{\binom{2k}{k}k^{2}}.
\]
With $\alpha_{k}=\frac{k}{2\left(2k+1\right)}$ and $\beta_{k}=\frac{3}{2k},$
we have
\[
\sum_{n\geqslant1}\left(\prod_{k=1}^{n-1}\alpha_{k}\right)\beta_{n}=\frac{3}{2}\sum_{n\geqslant1}\frac{2}{n^{2}\binom{2n}{n}}=\zeta\left(2\right).
\]
\end{proof}
Identifying the linear term in (\ref{eq:Borwein}) produces
\[
\zeta\left(4\right)=3\sum_{k\geqslant1}\frac{1}{\binom{2k}{k}k^{2}}\left(\frac{1}{k^{2}}-3H_{k-1}^{\left(2\right)}\right).
\]

This suggests the following result.
\begin{thm}
A Gosper representation for $\zeta\left(2\right)$ and $\zeta\left(4\right)$
is obtained as
\[
\prod_{n\geqslant1}\left(\begin{array}{ccc}
\frac{n}{2\left(2n+1\right)} & \frac{-3}{2n\left(2n+1\right)} & \frac{3}{2n^{3}}\\
0 & \frac{n}{2\left(2n+1\right)} & \frac{3}{2n}\\
0 & 0 & 1
\end{array}\right)=\left(\begin{array}{ccc}
0 & 0 & \zeta\left(4\right)\\
0 & 0 & \zeta\left(2\right)\\
0 & 0 & 1
\end{array}\right).
\]
\end{thm}

\begin{proof}
Denote
\[
M_{n}=\left(\begin{array}{ccc}
\delta_{n} & \gamma_{n} & u_{n}^{\left(1\right)}\\
0 & \delta_{n} & u_{n}^{\left(2\right)}\\
0 & 0 & 1
\end{array}\right)=\left(\begin{array}{cc}
A_{n} & \mathbf{u}_{n}\\
\mathbf{0} & 1
\end{array}\right)
\]
with $A_{n}=\left(\begin{array}{cc}
\delta_{n} & \gamma_{n}\\
0 & \delta_{n}
\end{array}\right)=\delta_{n}I+\gamma_{n}J$ and $\delta_{n}=\frac{2}{n\left(2n+1\right)}$ so that, with $I=\left[\begin{array}{cc}
1 & 0\\
0 & 1
\end{array}\right],J=\left[\begin{array}{cc}
0 & 1\\
0 & 0
\end{array}\right],\mathbf{u}_{n}=\left[\begin{array}{c}
u_{n}^{\left(1\right)}\\
u_{n}^{\left(2\right)}=\frac{3}{2n}
\end{array}\right],$
\[
A_{1}\dots A_{i-1}=\frac{2}{i\binom{2i}{i}}\left[I+J\sum_{j=1}^{i-1}\frac{\gamma_{j}}{\delta_{j}}\right].
\]
We know that
\[
M_{1}\dots M_{n}=\left(\begin{array}{cc}
A_{1}\dots A_{n} & \mathbf{v}_{n}\\
\mathbf{0} & 1
\end{array}\right)
\]
with 
\[
\mathbf{v}_{n}=\sum_{i=1}^{n}A_{1}\dots A_{i-1}\mathbf{u}_{i}
\]
so that
\begin{align*}
\mathbf{v}_{n} & =\sum_{i=1}^{n}\frac{2}{i\binom{2i}{i}}\left[\mathbf{u}_{i}+\sum_{j=1}^{i-1}\frac{\gamma_{j}}{\delta_{j}}\left[\begin{array}{c}
\frac{3}{2i}\\
0
\end{array}\right]\right]=\sum_{i=1}^{n}\frac{2}{i\binom{2i}{i}}\left[\left[\begin{array}{c}
u_{i}^{\left(1\right)}\\
\frac{3}{2i}
\end{array}\right]+\sum_{j=1}^{i-1}\frac{\gamma_{j}}{\delta_{j}}\left[\begin{array}{c}
\frac{3}{2i}\\
0
\end{array}\right]\right].\\
 & =\left[\begin{array}{c}
\sum_{i=1}^{n}\frac{2}{i\binom{2i}{i}}u_{i}^{\left(1\right)}+\sum_{j=1}^{i-1}\frac{\gamma_{j}}{\delta_{j}}\frac{3}{2i}\\
\sum_{i=1}^{n}\frac{2}{i\binom{2i}{i}}\frac{3}{2i}
\end{array}\right]
\end{align*}
This produces
\[
v_{\infty}^{\left(2\right)}=\zeta\left(2\right)=\sum_{i=1}^{\infty}\frac{3}{i^{2}\binom{2i}{i}}
\]
and
\[
v_{\infty}^{\left(1\right)}=\zeta\left(4\right)=\sum_{i=1}^{\infty}\frac{2}{i\binom{2i}{i}}u_{i}^{\left(1\right)}+\sum_{i=1}^{\infty}\frac{2}{i\binom{2i}{i}}\frac{3}{2i}\sum_{j=1}^{i-1}\frac{\gamma_{j}}{\delta_{j}}.
\]
Identifying with 
\[
\zeta\left(4\right)=3\sum_{k\geqslant1}\frac{1}{\binom{2k}{k}k^{2}}\left(\frac{1}{k^{2}}-3H_{k-1}^{\left(2\right)}\right)
\]
produces
\[
u_{i}^{\left(1\right)}=\frac{3}{2i^{3}},\thinspace\thinspace\gamma_{j}=\frac{-3}{2j\left(2j+1\right)}.
\]
\end{proof}
Unfortunately, the case that includes $\zeta\left(6\right)$ is not
as straightforward.
\begin{thm}
A Gosper representation for $\zeta\left(2\right),\thinspace\thinspace\zeta\left(4\right)$
and $\zeta\left(6\right)$ is obtained as
\[
\prod_{n\geqslant1}\left[\begin{array}{cccc}
\frac{n}{2\left(2n+1\right)} & -\frac{3}{2n\left(2n+1\right)} & 0 & \frac{3}{2n^{5}}-\frac{9H_{n-1}^{\left(4\right)}}{2n}\\
0 & \frac{n}{2\left(2n+1\right)} & -\frac{3}{2n\left(2n+1\right)} & \frac{3}{2n^{3}}\\
0 & 0 & \frac{n}{2\left(2n+1\right)} & \frac{3}{2n}\\
0 & 0 & 0 & 1
\end{array}\right]=\left[\begin{array}{cccc}
0 & 0 & 0 & \zeta\left(6\right)\\
0 & 0 & 0 & \zeta\left(4\right)\\
0 & 0 & 0 & \zeta\left(2\right)\\
0 & 0 & 0 & 1
\end{array}\right].
\]
For example, the truncated product from $n=1$ up to $n=200$ is 
\[
\left[\begin{array}{cccc}
\text{\ensuremath{2.4222.10^{-122}}} & \ensuremath{-1.1917.10^{-121}} & \text{\ensuremath{1.7517.10^{-121}}} & 1.01734\\
0. & \text{\ensuremath{2.4222.10^{-122}}} & \text{-\ensuremath{1.1917.10^{-121}}} & 1.08232\\
0. & 0. & \text{\ensuremath{2.4222.10^{-122}}} & 1.64493\\
0. & 0. & 0. & 1.
\end{array}\right].
\]
\end{thm}

\begin{proof}
Identifying the coefficient of $z^{2}$ in Borwein's identity (\ref{eq:Borwein})
produces
\[
\zeta\left(6\right)=3\sum_{k\geqslant1}\frac{1}{\binom{2k}{k}k^{2}}\left[17H_{k-1}^{\left(2,2\right)}+H_{k-1}^{\left(4\right)}-4\left(H_{k-1}^{\left(2\right)}\right)^{2}-\frac{3H_{k-1}^{\left(2\right)}}{k^{2}}+\frac{1}{k^{4}}\right].
\]
Moreover, the vector $\mathbf{v}_{n}$ is computed as 
\[
\mathbf{v}_{n}=\sum_{i=1}^{n}A_{1}\dots A_{i-1}\mathbf{u}_{i}=\sum_{i=1}^{n}\frac{2}{\binom{2i}{i}i}\left[\mathbf{u}_{i}-3H_{i-1}^{\left(2\right)}\left[\begin{array}{c}
u_{i}\left(2\right)\\
u_{i}\left(3\right)\\
0
\end{array}\right]+9H_{i-1}^{\left(2,2\right)}\left[\begin{array}{c}
u_{i}\left(3\right)\\
0\\
0
\end{array}\right]\right].
\]
Hence
\begin{align*}
v_{\infty}^{\left(1\right)} & =\sum_{i=1}^{\infty}\frac{2}{\binom{2i}{i}i}\left[\mathbf{u}_{i}\left(1\right)-3H_{i-1}^{\left(2\right)}u_{i}\left(2\right)+9H_{i-1}^{\left(2,2\right)}u_{i}\left(3\right)\right]\\
 & =\sum_{i=1}^{\infty}\frac{2}{\binom{2i}{i}i}\left[\mathbf{u}_{i}\left(1\right)-3H_{i-1}^{\left(2\right)}\frac{3}{2i^{3}}+9H_{i-1}^{\left(2,2\right)}\frac{3}{2i}\right].
\end{align*}
Using 
\[
\frac{3}{n}\left(4H_{n-1}^{\left(2,2\right)}+\frac{1}{2}H_{n-1}^{\left(4\right)}-2\left(H_{n-1}^{\left(2\right)}\right)^{2}\right)=-\frac{9}{2n}H_{n-1}^{\left(4\right)}
\]
and identifying $v_{\infty}^{\left(1\right)}=\zeta\left(6\right)$
produces the result.
\end{proof}

\subsection{A finite matrix product for $H_{N}^{\left(3\right)}$}
From the identity
\[
H_{N}^{\left(3\right)}=\sum_{n=1}^{N}\frac{\left(-1\right)^{n-1}}{n^{3}\binom{2n}{n}}\left[\frac{5}{2}-\frac{1}{2\binom{N+n}{2n}}\right],
\]
we deduce the following finite product representation
\[
\prod_{n=1}^{N}\left[\begin{array}{cc}
-\dfrac{n}{2\left(2n+1\right)} & \dfrac{5}{4n^{2}}\left(1-\dfrac{1}{5\binom{N+n}{2n}}\right)\\
0 & 1
\end{array}\right]=\left[\begin{array}{cc}
\dfrac{2\left(-1\right)^{N}}{\left(N+1\right)\binom{2N+2}{N+1}} & H_{N}^{\left(3\right)}\\
0 & 1
\end{array}\right].
\]

\section{Gosper Representation of Markov's identity for $\zeta(2)$ and $\zeta(z+1,3)$}
\subsection{Markov's identity for $\zeta(z+1,3)$} Markov's identity reads
\begin{equation}
\zeta\left(z+1,3\right)=\sum_{n=1}^{\infty}\frac{1}{\left(n+z\right)^{3}}=\frac{1}{4}\sum_{k=1}^{\infty}\frac{\left(-1\right)^{k-1}\left(k-1\right)!^{6}}{\left(2k-1\right)!}\frac{5k^{2}+6kz+2z^{2}}{\left(\left(z+1\right)\left(z+2\right)\ldots\left(z+k\right)\right)^{4}}.
\label{Markov}\end{equation}
\begin{thm}
A Gosper's representation for Markov's identity is
\[
\prod_{n=1}^{\infty}\left[\begin{array}{cc}
-\dfrac{n^{6}}{2n\left(2n+1\right)\left(z+n+1\right)^{4}} & 5k^{2}+6kz+2z^{2}\\
0 & 1
\end{array}\right]=\left[\begin{array}{cc}
0 & 4\left(z+1\right)^{4}\zeta\left(z+1,3\right)\\
0 & 1
\end{array}\right]
\]
or equivalently
\[
\prod_{n=1}^{\infty}\left[\begin{array}{cc}
-\dfrac{n^{6}}{2n\left(2n+1\right)\left(z+n+1\right)^{4}} & \dfrac{5k^{2}+6kz+2z^{2}}{4\left(z+1\right)^{4}}\\
0 & 1
\end{array}\right]=\left[\begin{array}{cc}
0 & \zeta\left(z+1,3\right)\\
0 & 1
\end{array}\right].
\]
\end{thm}

\begin{proof}
Rewrite Markov's identity as
\[
4\zeta\left(z+1,3\right)=\sum_{k\geqslant1}\frac{\left(-1\right)^{k-1}\left(k-1\right)!^{6}}{\left(2k-1\right)!}\frac{5k^{2}+6kz+2z^{2}}{\left(\left(z+1\right)\dots\left(z+k\right)\right)^{4}},
\]
define
\[
u_{k}=5k^{2}+6kz+2z^{2}
\]
and notice that writing
\[
4\zeta\left(z+1,3\right)=u_{1}+\alpha_{1}u_{2}+\alpha_{1}\alpha_{2}u_{3}+\dots
\]
requires that the coefficient of $u_{1}$ should be equal to $1;$
as it is equal to $\frac{1}{\left(z+1\right)^{4}},$ consider the
variation
\[
4\left(z+1\right)^{4}\zeta\left(z+1,3\right)=\sum_{k\geqslant1}\frac{\left(-1\right)^{k-1}\left(k-1\right)!^{6}}{\left(2k-1\right)!}\frac{\left(z+1\right)^{4}}{\left(\left(z+1\right)\dots\left(z+k\right)\right)^{4}}u_{k}
\]
which now satisfies this constraint. Then identifying
\[
\alpha_{1}\dots\alpha_{k-1}=\frac{\left(-1\right)^{k-1}\left(k-1\right)!^{6}}{\left(2k-1\right)!}\frac{\left(z+1\right)^{4}}{\left(\left(z+1\right)\dots\left(z+k\right)\right)^{4}}
\]
provides
\[
\alpha_{k}=\frac{-k^{6}}{2k\left(2k+1\right)\left(z+k+1\right)^{4}}.
\]
Notice that the constant term $\left(z+1\right)^{4}$ disappears from
$\alpha_{k}.$
\end{proof}
Another identity \cite{Tauraso} due to Tauraso is
\begin{equation}
\sum_{n\geqslant1}\frac{1}{n^{2}-an-b^{2}}=\sum_{k\geqslant1}\frac{3k-a}{\binom{2k}{k}k}\frac{1}{k^{2}-ak-b^{2}}\prod_{j=1}^{k-1}\frac{j^{2}-a^{2}-4b^{2}}{j^{2}-aj-b^{2}}.\label{eq:Tauraso}
\end{equation}

\begin{thm}
A Gosper's matrix representation for identity (\ref{eq:Tauraso})
is

\[
\prod_{n=1}^{\infty}\left[\begin{array}{cc}
\frac{k}{2\left(2k+1\right)}\frac{k^{2}-a^{2}-4b^{2}}{k^{2}-ak-b^{2}} & \frac{3k-a}{k^{2}-ak-b^{2}}\\
0 & 1
\end{array}\right]=\left[\begin{array}{cc}
0 & \sum_{n\geqslant1}\frac{2}{n^{2}-an-b^{2}}\\
0 & 1
\end{array}\right].
\]
Notice that
\[
\sum_{n\geqslant1}\frac{2}{n^{2}-an-b^{2}}=\frac{2}{\sqrt{a^{2}+4b^{2}}}\left[\psi\left(1-\frac{a}{2}+\frac{\sqrt{a^{2}+4b^{2}}}{2}\right)-\psi\left(1-\frac{a}{2}-\frac{\sqrt{a^{2}+4b^{2}}}{2}\right)\right].
\]
\end{thm}

\begin{proof}
Choose
\[
u_{k}=\frac{3k-a}{k^{2}-ak-b^{2}}.
\]
The first term in (\ref{eq:Tauraso}) is
\[
\frac{3-a}{2}\frac{1}{1-a-b^{2}}=\frac{1}{2}u_{1}
\]
so that we consider twice identity (\ref{eq:Tauraso}), and choose
\[
\alpha_{1}\dots\alpha_{k-1}=\frac{1}{\binom{2k}{k}k}\prod_{j=1}^{k-1}\frac{j^{2}-a^{2}-4b^{2}}{j^{2}-aj-b^{2}}
\]
so that
\[
\alpha_{k}=\frac{k^{2}-a^{2}-4b^{2}}{k^{2}-ak-b^{2}}\frac{k}{2\left(2k+1\right)}.
\]
\end{proof}
A quartic version reads
\[
\sum_{n\geqslant1}\frac{n}{n^{4}-a^{2}n^{2}-b^{4}}=\frac{1}{2}\sum_{k\geqslant1}\frac{\left(-1\right)^{k-1}}{\binom{2k}{k}k}\frac{5k^{2}-a^{2}}{k^{4}-a^{2}k^{2}-b^{4}}\prod_{j=1}^{k-1}\frac{\left(j^{2}-a^{2}\right)^{2}+4b^{4}}{j^{4}-a^{2}j^{2}-b^{4}}.
\]
The same approach as above produces
\[
\prod_{n=1}^{\infty}\left[\begin{array}{cc}
-\frac{k}{2\left(2k+1\right)}\frac{\left(k^{2}-a^{2}\right)^{2}+4b^{4}}{k^{4}-a^{2}k^{2}-b^{4}} & \frac{5k^{2}-a^{2}}{k^{4}-a^{2}k^{2}-b^{4}}\\
0 & 1
\end{array}\right]=\left[\begin{array}{cc}
0 & \sum_{n=1}\frac{4n}{n^{4}-a^{2}n^{2}-b^{4}}\\
0 & 1
\end{array}\right].
\]

Amdeberhan-Zeilberger's ultra-fast series representation \cite{Amdeberhan}
\[
\zeta\left(3\right)=\sum_{n\geqslant1}\left(-1\right)^{n-1}\frac{\left(n-1\right)!^{10}}{64\left(2n-1\right)!^{5}}\left(205n^{2}-160n+32\right)
\]
can be realized as
\[
\prod_{n=1}^{\infty}\left[\begin{array}{cc}
-\left(\frac{k}{2\left(2k+1\right)}\right)^{5} & 205k^{2}-160k+32\\
0 & 1
\end{array}\right]=\left[\begin{array}{cc}
0 & 64\zeta\left(3\right)\\
0 & 1
\end{array}\right]
\]
or equivalently
\[
\prod_{n=1}^{\infty}\left[\begin{array}{cc}
-\left(\frac{k}{2\left(2k+1\right)}\right)^{5} & \frac{205k^{2}-160k+32}{64}\\
0 & 1
\end{array}\right]=\left[\begin{array}{cc}
0 & \zeta\left(3\right)\\
0 & 1
\end{array}\right].
\]

The resemblance with (\ref{eq:product 3}) is interesting and suggests
the generalization
\[
\prod_{n=1}^{\infty}\left[\begin{array}{ccc}
-\left(\frac{k}{2\left(2k+1\right)}\right)^{5} & \left(\frac{1}{2k\left(2k+1\right)}\right)^{5} & P\left(k\right)\\
0 & -\left(\frac{k}{2\left(2k+1\right)}\right)^{5} & \frac{205k^{2}-160k+32}{64}\\
0 & 0 & 1
\end{array}\right]=\left[\begin{array}{ccc}
0 & 0 & \zeta\left(5\right)\\
0 & 0 & \zeta\left(3\right)\\
0 & 0 & 1
\end{array}\right]
\]
where $P\left(k\right)$ is to be determined.
Another fast representation due to Amdeberhan \cite{Ambdeberhan 2}
is
\[
\zeta\left(3\right)=\frac{1}{4}\sum_{n=1}^{\infty}\frac{\left(-1\right)^{n-1}\left(56n^{2}-32n+5\right)}{n^{3}\left(2n-1\right)^{2}\binom{3n}{n}\binom{2n}{n}}
\]
and produces
\[
\prod_{n=1}^{\infty}\left[\begin{array}{cc}
-\dfrac{k^{3}}{\left(3k+3\right)\left(3k+2\right)\left(3k+1\right)}\left(\dfrac{2k-1}{2k+1}\right)^{2} & \dfrac{56k^{2}-32k+5}{24}\\
0 & 1
\end{array}\right]=\left[\begin{array}{cc}
0 & \zeta\left(3\right)\\
0 & 1
\end{array}\right].
\]

\end{document}